\documentclass[12pt,draft,amsmath,amscd,amsbsy,amssymb]{amsart}
\usepackage{amssymb,latexsym, amsmath, amscd, array, graphicx}

\swapnumbers \numberwithin{equation}{section}

\theoremstyle{plain}

\newtheorem{thm}{Theorem}[section]

\newtheorem{conjec}[thm]{Conjecture}
\newtheorem{prop}[thm]{Proposition}

\theoremstyle{definition}

 \newcommand{\Wi}{\widetilde}

\def\C{{\mathbb C}}
\def\Z{{\mathbb Z}}
\def\Q{{\mathbb Q}}

\def\N{{\mathbb N}}

\def\1{\hbox{\rm\rlap {1}\hskip.03in{\rom I}}}
\def\Bbbone{{\rm1\mathchoice{\kern-0.25em}{\kern-0.25em}
{\kern-0.2em}{\kern-0.2em}I}}

\long\def\forget#1\forgotten{} %

\newcommand\ver[1]{\marginpar{\tiny Changed in Ver \VER}}

\date{\today}
\begin{document}

\title[Burghelea Conjecture]{On Burghelea Conjecture }

\author[A.~Dranishnikov]{Alexander  Dranishnikov} %

\thanks{The author was supported by NSF DMS-1304627}

\begin{abstract}
We prove the Burghelea Conjecture for all groups with finite  cohomological dimension satisfying some additional condition on cohomology
of reduced normalizers.
\end{abstract}

\address{Alexander N. Dranishnikov, Department of Mathematics, University
of Florida, 358 Little Hall, Gainesville, FL 32611-8105, USA}
\email{dranish@math.ufl.edu}


\maketitle

\section{Introduction}
We recall that for an algebra $A$ over a field $k$ (generally, over a commutative ring) $HH_*(A)$ and $HC_*(A)$
denote the Hochschild homology and the cyclic homology of $A$~\cite{L},\cite{R}. These two homology groups are connected by the
Connes exact sequence
$$
\dots\to HC_{n-1}(A) \to HH_n(A) \to HC_n(A)\stackrel{S}\longrightarrow HC_{n-2}(A) \to\dots .
$$
We use notation $PHC_*(A)$ for the periodic cyclic homology, $*=0,1$. 
The shift homomorphism $S$ in the Connes exact sequence defines the inverse sequence $\{HC_{*+2n},S\}_{n\in \N}$.
We note that for the periodic cyclic homology there is an epimorphism
$
PHC_*(A)\to\lim_{\leftarrow}\{HC_{*+2n}(A),S\}
$
with the $lim^1$-kernel. The kernel is trivial if the Mittag-Lefler condition is satisfied in particular when $HH_m(A)=0$ for all 
sufficiently large $m$~(see \cite{R} Corollary 6.1.23).

We consider the case when $k=\Q$ and $A$ is the group algebra $\Q G$ for a discrete group $G$. It was known~\cite{ML} 
well before the invention of the cyclic homology in early 80s (\cite{C},\cite{LQ},\cite{T}) that
$HH_n(\Q G)=H_n(G,\widetilde{\Q G})$ where $\widetilde{\Q G}$ is the $G$-module $\Q G$ defined 
by the $G$-action on $G$ by conjugation. Let $\langle G\rangle$ denote the set of conjugacy classes on $G$
and let $G_x$ denote the centralizer of $x\in[x]\in\langle G\rangle$. Clearly, $G_y\cong G_z$ for $y,z\in[x]$. Since 
$$\widetilde{\Q G}=\bigoplus_{[x]\in\langle G\rangle}\Q(G/G_x)=\bigoplus_{[x]\in\langle  G\rangle}\Q\otimes_{G_x}\Q G,$$
we obtain 
$$
HH_n(\Q G)=\bigoplus_{[x]\in\langle G\rangle}H_n(G,\Q\otimes_{G_x}\Q G).$$ By the Shapiro Lemma~\cite{Br},
$H_n(G,\Q\otimes_{G_x}\Q G)=H_n(G_x,\Q)$. Thus,
$$
HH_n(\Q G)=\bigoplus_{[x]\in\langle  G\rangle}H_n(G_x,\Q).$$

Burghelea have computed cyclic  homology of group algebras~\cite{Bu}
$$
HC_n(\Q G)=\bigoplus_{{\langle G\rangle}^{fin}}[H_*(N_x,\Q)\otimes 
H_*(\C P^\infty;\Q)]_n\oplus \bigoplus_{{\langle G\rangle}^{\infty}}H_n(N_x,\Q)
$$
where $N_x=G_x/\langle x\rangle$ is the reduced centralizer, ${\langle G\rangle}^{fin}$ is the subset of $\langle G\rangle$ of conjugacy classes of elements of finite order,
and ${\langle G\rangle}^{\infty}$ is the subset of conjugacy classes of elements of infinite order. This computation brought the following formula
for the periodic cyclic homology of $\Q G$~\cite{Bu}
$$
PHC_0(\Q G)=\bigoplus_{[x]\in{\langle G\rangle}^{fin},n\ge 0}H_{2n}(N_x,\Q)\oplus \bigoplus_{{\langle G\rangle}^{\infty}}T_0(x,\Q)
$$
$$
PHC_1(\Q G)=\bigoplus_{[x]\in{\langle G\rangle}^{fin}, n\ge 0}H_{2n+1}(N_x,\Q)\oplus \bigoplus_{{\langle G\rangle}^{\infty}}T_1(x,\Q).
$$
Here the group $T_*(x,\Q)$ fits in the following short exact sequence
$$
0\to{\lim_{\leftarrow}}^1\{H_{*-1+2n}(N_x,\Q)\}\to T_*(x,\Q)\to\lim_{\leftarrow}\{H_{*+2n}(N_x,\Q)\}\to 0
$$
where bonding maps in the inverse sequences are the Gysin homomorphisms $S:H_m(N_x,\Q)\to H_{m-2}(N_x,\Q)$ coming from the fibration
$S^1\to BG_x\to BN_x$.

In the end of his paper~\cite{Bu} Burghelea stated the following conjecture:
\begin{conjec}[Burghelea]
If $K(G,1)$ has the homotopy type of a finite CW-complex, then $T_*(x,\Q)=0$ for all $x\in{\langle G\rangle}^{\infty}$.
\end{conjec}
Burghelea proved his conjecture for the fundamental groups of  Riemannian manifolds with nonpositive sectional curvature.
He also noticed that not every group satisfies the condition $T_*(x,\Q)=0$ for all $x\in{\langle G\rangle}^{\infty}$.
We say that a group $G$ satisfies the Burghelea Conjecture if it satisfies this condition
for all $x\in{\langle G\rangle}^{\infty}$. Eckmann proved~\cite{E} the Burghelea Conjecture
for groups $G$ with $hd_\Q G<\infty$ in the case when $G$ is nilpotent or torsion free solvable group. Also he proved the Burghelea Conjecture for
linear groups over fields of characteristic 0 and for groups with $cd_\Q G\le 2$. R. Ji~\cite{J} proved the Burghelea Conjecture for hyperbolic groups, arithmetic groups, and groups of polynomial growth.

In this paper we  prove the following theorem. 
\begin{thm}\label{dm}
Suppose that a group $G$ satisfies  $cd(G)<\infty$ and all reduced cetralizers $N_x$, $x\in G$, have property trivial rational homology $H_i(N_x;\Q)=0$
for infinitely many $i$. Then  $T_*(x,\Q)=0$ for all $x\in G$.
\end{thm}

Eckmann~\cite{E} (see also~\cite{Em})  noticed that the Hattori-Stallings rank map $$r:K_0(\Q G)\to HC_0(\Q G)=\Q G/[\Q G,\Q G]=\bigoplus_{\langle G\rangle}\Q$$ (see~\cite{Br} for the definition)
can be factored through the Connes-Karubi character $$ch^n:K_0(kG)\to HC_{2n}(kG), \ \ \ n\ge 0,$$ as $r=S^n\circ ch^n$ where $S^n=S\circ\cdots\circ S$ is the $n$ times iterated shift homomorphism. It follows from a result of Linnel~\cite{Li}, that the Hattori-Stallings rank takes value in $\Q\oplus\bigoplus_{\langle G\rangle^\infty}\Q$ where the first summand corresponds to the conjugacy class of the unit $e\in G$. Therefore, the Hattori-Stallings rank factors thorough $\Q\oplus\bigoplus_{\langle G\rangle^\infty}T_0(x,\Q)$.
Thus, for fixed $G$ the  Burghelea Conjecture  implies the Bass Conjecture~\cite{Ba}  on vanishing the terms 
of the image of the Hattori-Stallings map in $\bigoplus_{\langle G\rangle^\infty}\Q$. 

Note that in  it's turn the Bass Conjecture for $G$ implies 
the Idempotents Conjecture~\cite{Em}: {\em In the group algebra $\Q G$ for a torsion free group $G$ the equation $x^2=x$ has  only trivial solutions $x=0$ and $x=1$.}

\

The author is thankfull  to Alex Engel, Michal Marcinkowski, and Oscar Randal-Williams for pointing out on mistakes in the early versions
of the paper and giving valuable remarks.

\section{Free loop spaces}
Let $LX$ denote the free loop space on $X$, $LX=Map(S^1,X)$.
By $L_0X$ we denote the space of free null-homotopic loops on $X$, $L_0X\subset LX$. 
When $X$ is simply connected we have $L_0X=LX$.
The group $S^1$ naturally acts on $LX$
 by $(zf)(u)=f(z^{-1}u)$. Let $S(X)$ denote the orbit space $LX/S^1$, the space of  'strings'. 
Note that $L_0X$ is $S^1$-invariant and denote by
$S_0(X)$ the orbit space $L_0X/S^1$, the space of null-homotopic 'strings'
on $X$. 

There is the evaluation  fibration
$ev:LX\to X$, $ev(\phi:S^1\to X)=\phi(1)$. Note that  $ev$ admits a section $s:X\to LX$ and the fiber of $ev$
is the loop space $\Omega X$.

We note that $L$,  $L_0$, $S$, and $S_0$ are covariant functors on the category of topological spaces.
Each of these functors $F$ takes a homotopy to a homotopy in a sense that there is a natural embedding
$F(X)\times I\to F(X\times I)$.

Suppose that $X$ is  a proper metric space which uniformly locally path connected, i.e. there are $\epsilon>0$ and a continuous map $\Phi:W_{\epsilon}\times[0,1]\to X$ of the $\epsilon$-neighborhood
$W_{\epsilon}$ of the diagonal $\Delta(X)\subset X\times X$ such that $\Phi(x,y,-):[0,1]\to X$ is a path from $x$ to $y$ and $\Phi(x,x,-)$ 
is a constant path for all $(x,y)\in W_{\epsilon}$.
Then  $L_0X$ is uniformly locally path connected by means of the function $\Wi\Phi:\Wi W_{\epsilon}\times[0,1]\to L_0X$ defined
as $\Wi\Phi(f_1,f_2,t)(u)=\Phi(f_1(u),f_2(u),t)$ where $\Wi W_{\epsilon}$ is the epsilon neighborhood of the diagonal $\Delta(L_0X)$ in the uniform metric.
We note that $\Wi W_{\epsilon}$ is invariant with respect to the diagonal $S^1$-action and the map $\Wi\Phi$ is equivariant. Then the space
$S_0(X)$ is locally contractible. Moreover, the space $S_0(X)/X$ is locally contractible. Note that for such spaces the singular cohomology behaves well
and agrees with the \v Cech cohomology~\cite{Bre}.

We note that $X$ is embedded in $LX$ as well in $S_0X$ via the constant loops. Let $j_X:X\to S_0(X)$ denote that embedding.
Note that  $j_X:1\to S_0$ is a natural transformation of functors.
\begin{prop}
For any space $X$
$$
SX/X=(LX/X)/S^1\ \ \ \ \ \text{and}\ \ \ \ S_0X/X=(L_0X/X)/S^1
$$
where we identify $X$ with $j_X(X)$.
\end{prop}
\begin{proof}
This is an obvious statement on the set-theoretical level: Both spaces consists of orbits of non-constant loops $f:S^1\to X$ plus the special point $\ast=\{X\}$.
The rest of the proof is an exercise on the quotient topology which we feel obligated to do since the spaces involved are not compact. 

We recall that the main theorem on quotient map which states that for a quotient map $p:X\to Y$ and a continuous map $g:X\to Z$ constant on each point preimage $p^{-1}(y)$ there is a continuous map $f:Y\to Z$ such that $f\circ p=g$. We apply it to the quotient map $p:LX/X\to (LX/X)/S^1$ and $g:LX/X\to SX/X$ t and to the quotient map $LX\to SX\to SX/X$ and a map $LX\to LX/X\to (LX/X)/S^1$ to obtain a map $f:(LX/X)/S^1\to SX/X$ and its inverse $f^{-1}:SX/X\to (LX/X)/S^1$.
Here we used the fact that the composition of quotient maps is a quotient map.
\end{proof}
We consider the Borel construction~\cite{Bo} for the $S^1$ action on $L_0X/X$ and on $S^{2N+1}$ for sufficiently large $N$:
$$
\begin{CD}
S^{2N+1} @<<< S^{2N+1}\times(L_0X/X) @>>> L_0X/X\\
@VVV @VVV @VVV\\
\C P^N @< p_1<< S^{2N+1}\times_{S^1}(L_0X/X) @>p_2>> S_0(X)/X .\\
\end{CD}
$$
Every orbit $S^1z\in S_0(X)/X$ is homeomorphic to to $S^1/G_z$ where $G_z\subset S^1$ is the stabilizer of $z$, a closed subgroup of $S^1$.
Then the fiber $p_2^{-1}(z)$ is homeomorphic to $S^{2N+1}/G_z$.
We note that the special fiber $F=p_2^{-1}(\{X\})$ of $p_2$ is homeomorphic to $\C P^N$. We denote by $$E(N,X)= (S^{2N+1}\times_{S^1}(L_0X/X))/F$$
the quotient space and by $\bar p_2:E(N,X)\to S_0(X)/X$ the induced map. Thus $p_2=\bar p_2\circ q$ where $q$ is the map collapsing $F$.
\begin{prop}\label{Borel} 
For a uniformly locally path connected proper metric space $X$ the projection $\bar p_2$ induces an isomorphism
$$
(\bar p_2)_*:H_i(E(N,X);\Q) \to H_i(S_0(X)/X;\Q)
$$
in dimensions $i\le 2N$.
\end{prop}
\begin{proof}
Since the spaces here are locally nice, it suffices to prove the statement for the cohomology. We consider the Leray spectral sequence of $\bar p_2$.
Since the map $\bar p_2$ is proper, the sequence converges to $H^*(E(N,X);\Q)$~\cite{Bre}. Note that $\bar p_2$ has one exceptional fiber
which is a point. All other fibers are  lens spaces $L^{2N+1}_m$ possibly degenerated to the sphere $S^{2N+1}$. Therefore,  in the Leray spectral sequence we have trivial stalks 
$$\mathcal H^q(\bar p_2)=\lim_\rightarrow H^q(\bar p_2^{-1}(U);\Q)=H^q(\bar p_2^{-1}(x);\Q)=0$$ for $0<q\le 2N$. Hence, $E_2^{p,q}=0$ 
for $0<q\le 2N$ and the result follows.
\end{proof}

\begin{prop}\label{sphere} Let $X=B\pi/B\pi^{(n-1)}$ and $H_{n+1}(B\pi;\Q)=0$, $n\ge 4$. Then $H_n(LX/X;\Q)= 0$.
\end{prop}
\begin{proof} 
Since $X$ is $(n-1)$-connected and $\Omega(X)$ is $(n-2)$-connected  in the homology Leray-Serre spectral sequence of $ev:LX\to X$ we have  $E^2_{p,q}=0$  for $0<q<n-1$ and 
$E^2_{p,0}=0$ for $0<p<n$. This implies that $H_n(LX)$ is defined by $E^2_{0,n}=H_n(X)$ and $E^2_{n,0}$. 
Since $ev$ has a section, it implies that both terms live to $E^\infty$. 
The splitting $H_n(LX)= H_n(X)\oplus H_n(LX/X)$ generated by the section $s$
implies that $H_n(LX/X;\Q)=H_n(\Omega X;\Q)$. Note that $H_n(\Omega X;\Q)=\pi^s_n(\Omega X)\otimes\Q=\pi_n(\Omega X)\otimes\Q=\pi_{n+1}(X)\otimes\Q=\pi^s_{n+1}(X)\otimes\Q=H_{n+1}(X;Q)=H_{n+1}(B\pi;\Q)=0$.
\end{proof}

\begin{prop}\label{Leray}
Let $\phi_{\ast,\ast}:E_{\ast,\ast}(p')\to E_{\ast,\ast}(p)$ be the morphism of the Leray-Serre spectral sequences
generated by a morphism $(f',f):p'\to p$ of fiber bundles, $p':X'\to Y'$ and $p:X\to Y$. 

(a) If $\phi^2_{0,n}:E_{0,n}^2(p')\to E_{0,n}^2(p)$ is surjective then also is $\phi^\infty_{0,n}:E^\infty_{0,n}(p')\to E_{0,n}^\infty(p)$.

(b) If all $\phi^{\infty}_{k,l}:E_{k,l}^{\infty}(p')\to E^{\infty}_{k,l}(p)$ for $k+l=n$ are surjective,
then $f'_*:H_n(X')\to H_n(X)$ is surjective.

\end{prop}
\begin{proof} 
(a) We show by induction on $r$ that $\phi^r_{0,n}$ is surjective.
Indeed, since in the commutative diagram
$$
\begin{CD}
@>d^r>> E^r_{0,n}(p')@>>> E^{r+1}_{0,n}(p') @>>> 0\\
@. @V\phi^r_{0,n}VV @V\phi^{r+1}_{0,n}VV @.\\
@>d^r>> E^r_{0,n}(p)@>\psi>> E^{r+1}_{0,n}(p) @>>> 0\\
\end{CD}
$$
$\phi^r_{0,n}$ and $\psi$ are surjective, so is $\phi^{r+1}_{0,n}$.

(b) The homology groups $H_n(X)$ is obtained from $E^{\infty}_{n,0}(p)$ by consecutive extensions
by $E^{\infty}_{n-1,1}(p)$ then by $E^{\infty}_{n-2,2}(p)$ and so on up to $E^{\infty}_{0,n}$. The same holds true
for $H_n(X')$. Let $A_0= E^{\infty}_{n,0}(p)$, $A_0'= E^{\infty}_{n,0}(p')$ and let $A_i$ and $A_i'$ denote the corresponding intermediate extensions, $i=1,\dots, n$. Thus,
$H_n(X)=A_{n}$ and $H_n(X')=A_{n}'$.
We apply the epimorphism version of Five Lemma recursively to the diagram
$$
\begin{CD}
0 @>>> E^{\infty}_{n-i,i}(p') @>>> A_i ' @>>> A_{i-1}'@>>> 0\\
@.@VVV @VVV @VVV @.\\
0 @>>> E^\infty_{n-i,i}(p) @>>> A_i @>>> A_{i-1}@>>> 0\\
\end{CD}
$$
to obtain the result.
\end{proof}

The main result of this section, Theorem~\ref{strings}, is rather technical. It is a computation of the string homology in some special case. Since we are doing it by rather elementary means we don't involve here the loop homology and the string homology theories~\cite{CJY},\cite{We}.

\begin{thm}\label{strings}
Suppose that a group $\pi$ has a locally finite CW complex as a classifying space $B\pi$. 
Then the inclusion 

(a) $j_\pi:B\pi\to S_0(B\pi)$ induces an isomorphism 
$$(j_{\pi})_*:H_*(B\pi;\Q)\to H_*(S_0(B\pi);\Q)$$ of the rational homology groups.

(b) Suppose that  $H_{n+1}(B\pi;\Q)=0$, $n\ge 4$. Then the inclusion homomorphism $$H_n(S_0(B\pi^{(n-2)});\Q)\to H_n(S_0(B\pi));\Q)$$ is  zero.
\end{thm}
\begin{proof} 
(a) We show that $H_*(S_0(B\pi),j_{\pi}(B\pi);\Q)=0$.

Note that the  fiber  $\Omega(B\pi)$ of the evaluation fibration $ev:L(B\pi)\to B\pi$  is homotopy equivalent to $\pi$. This fibration admits a section $s$ and the path component of $s(B\pi)$ is exactly $L_0(B\pi)$. Thus, the restriction $p_0$ of $ev$ to $L_0(B\pi)$ is a fibration with homotopy trivial fiber. Hence the inclusion
$B\pi\to L_0(B\pi)$ is a homotopy equivalence. Thus, $L_0(B\pi)/B\pi$ is contractible. 
Therefore the projection in the Borel construction $$p_1:S^{2N+1}\times_{S^1}(L_0(B\pi)/B\pi)\to\C P^{N}$$ induces isomorphism of homology groups. 
By Proposition~\ref{Borel} the projection $$\bar p_2:(S^{2N+1}\times_{S^1}(L_0(B\pi)/B\pi)/F\to S_0(B\pi)/j_\pi(B\pi)$$ induces isomorphisms of rational homology in dimensions $\le 2N$. The exceptional fiber $F$ of $p_2$ is homeomorphic to $\C P^{N}$. It defines a section of $p_1$. Thus collapsing $F$ to a point  defines a contractible space. Therefore, $H_i(S_0(B\pi),j_{\pi}(B\pi);\Q)=0$ for $i\le 2N$. Since $N$ is arbitrary, the result follows.

(b) In view of the exact sequence of the
pair $(S_0(B\pi),S_0(B\pi^{(n-2)}))$ it suffices to show that the homomorphism $$q_*:H_n(S_0(B\pi);\Q)\to H_n(S_0(B\pi)/S_0(B\pi^{(n-2)});\Q)$$
generated by the collapsing map $$q:S_0(B\pi)\to S_0(B\pi)/S_0(B\pi^{(n-2)})$$ is injective. 
We note that the collapsing map $\psi:B\pi\to B\pi/B\pi^{(n-2)}$ induces an equivariant map $L_0(\psi):L_0(B\pi)\to L_0(B\pi/B\pi^{(n-2)})$ and, hence, it defines a map $s(\psi):S_0(B\pi)\to S_0(B\pi/B\pi^{(n-2)})$. Note that the map $s(\psi)$ factors through $q$, $s(\psi)=\psi'\circ q$.
Thus, it suffices to show that $s(\psi)$ is injective for $n$-homology. Since $(j_\pi)_*$ and $i$ are isomorphisms in the commutative diagram
$$
\begin{CD}
H_n(B\pi/B\pi^{(n-2)};\Q) @>j'_*>> H_n(S_0(B\pi/B\pi^{(n-2)});\Q)\\
@AiAA @As(\psi)AA\\
H_n(B\pi;\Q) @>(j_{\pi})_*>> H_n(S_0(B\pi);\Q) ,\\
\end{CD}
$$
it suffices to show that $j'_*$ is a monomorphism. 

Let $X= B\pi/B\pi^{(n-2)}$. We need to show that $$j'_*:H_n(X;\Q)\to H_n(S_0(X);\Q)$$ is a monomorphism.
The exact sequence of the pair $(S_0(X),X)$  reduces  the problem to showing that the collapsing map $q:S_0(X)\to S_0(X)/X$ induces an epimorphism
$q_*:H_{n+1}(S_0(X);\Q)\to H_{n+1}(S_0(X)/X;\Q)$.

We consider the commutative diagram generated by Borel's constructions for $S^1$-action on $LX$, $LX/X$, and $S^{2N+1}$ with $N>n$. 
$$
\begin{CD}
H_{n+1}(S^{2N+1}\times_{S^1}LX;\Q) @>(p_2')_*>> H_{n+1}(S_0X;\Q)\\
@V\bar q_*VV @Vq_*VV\\
H_{n+1}(S^{2N+1}\times_{S^1}(LX/X);\Q) @>(p_2)_*>> H_{n+1}((S_0X)/X;\Q) .\\
\end{CD}
$$
It suffices to show that both $(p_2)_*$ and $\bar q_*$ are surjective.
By Proposition~\ref{Borel}
the reduced projection $$ \bar p_2:(S^{2N+1}\times_{S^1}(LX/X))/p_2^{-1}(\ast)\to S_0X/X$$  induces isomorphism of the rational homology in dimensions $\le 2N$. 
Here $\ast=\{X\}$ is the orbit of the fixed point.
The projection $$p_1:S^{2N+1}\times_{S^1}(LX/X)\to \C P^{N}$$ is a locally trivial bundle with the fiber $LX/X$.  The fixed point $\ast$ of the 
$S^1$-action on $LX/X$ defines a section $s$ of $p_1$. Therefore, there is a splitting
$$
H_i(S^{2N+1}\times_{S^1}(LX/X))=H_i(\C P^N)\oplus H_i((S^{2N+1}\times_{S^1}(LX/X))/p_2^{-1}(\ast))
$$
generated by $p_1$ and the collapsing map. Hence $(p_2)_*$ is surjective.

Similarly any fixed point of the $S^1$ action on $LX$ defines a section of $p_1':S^{2N+1}\times_{S^1}(LX)\to \C P^{N}$.
We consider the morphism of the Leray-Serre spectral sequences $\pi_{\ast,\ast}:E_{\ast,\ast}(p'_1)\to E_{\ast,\ast}(p_1)$ of $p_1$ and $p_1'$ generated by $\bar q$.

By Proposition~\ref{sphere}, $H_{n-1}(LX/X;\Q)=0 $. Note that since in the definition of $X$ we collapse  the $(n-2)$-skeleton of $B\pi$ instead of the $(n-1)$-skeleton. Then  $E^2_{2,n-1}(p_1)=0$.
Therefore, nonzero elements in $E^2$-page for $p_1$ on the $(n+1)$-antidiagonal are only $E^2_{0,n+1}$ and $E^2_{n+1,0}$.
Since $H_{n+1}(LX;\Q)\to H_{n+1}(LX/X;\Q)$ is surjective, $\phi^2_{0,n+1}:E^2_{0,n}(p_1')\to E^2_{0,n}(p_1)$
is surjective. Clearly, $$\phi^2_{n+1,0}:E^2_{n+1,0}(p_1')=H_{n+1}(\C P^m;\Q)\to E^2_{n+1,0}(p_1)=H_{n+1}(\C P^m;\Q)$$ is an isomorphism.
 By Proposition~\ref{Leray} part (a), $\phi^{\infty}_{0,n+1}$
is surjective. Since both fibrations have sections, we have  stabilizations: $E^2_{n+1,0}(p_1)=E^{\infty}_{n+1,0}(p_1)$ and $E^2_{n+1,0}(p_1')=E^{\infty}_{n+1,0}(p_1')$.
Then by Proposition~\ref{Leray} part (b),  $\bar q_*$ is surjective in dimension $n+1$.
\end{proof}

\section{The main result}

Let $\nu:S^\infty\to\C P^\infty$ be the universal $S^1$-bundle. We identify
$$S^\infty=\{(E,v)\in \C P^\infty\times S^\infty\mid v\in E, |v|=1\}$$ with an $S^1$-invariant subset of $L(S^\infty)$ by sending $(E,v)$ to the rotation $f_v:S^1\to E$  with $f(1)=v$. Then $\C P^\infty$ is identified with a subset of $S(S^\infty)$. The same can be done for any principle $S^1$-bundle $p:E\to B$: The space $E$ can be identified with an invariant
subset of $L(E)$ and $B$ with a subset of $S(E)$. Let $\hat p:B\to S(E)$ denote this embedding.

By $cd(G)$ we denote the integral cohomological dimension of a group $G$. We recall that if $cd(G)\ne 2$ then
$cd(G)=\min\{\dim K(G,1)\}$~\cite{Br}.
\begin{thm}\label{main}
Suppose that in the central extension $$1\to C\to G\stackrel{p}\to H\to 1$$ where $C\cong \Z$ the group $G$ has finite $cd(G)<\infty$ and $H$ has 
trivial homology $H_n(H;\Q)$ for infinitely many $n$.
Then $H_n(BH;\Q)=0$ for all but finitely many $n$.
\end{thm}
\begin{proof}
 Assume that 
$H_n(BH;\Q)\ne 0$ and $H_{n+1}(BH;\Q)=0$ for some $n\ge cd(G)+2$, $n\ge 5$. 
Note that the homomorphism $\phi:G\to H$ is the induced homomorphism $p_*:\pi_1(BG)\to \pi_1(BH)$
for an oriented  $S^1$-fibration $p:BG\to BH$. Since every oriented $S^1$-bundle is a principal $S^1$-bundle~\cite{Mo}, $p$
is the pull-back of the universal $S^1$-bundle
$\nu: S^\infty\to\C P^\infty$.
 The composition 
$S(p)\circ\hat p:BH\to S(BH)$
coincides with  $j_H$ where $j_H:BH\to S_0(BH)\subset S(BH)$ is the inclusion map. 
 Note that $p:BG\to BH$ can be deformed to $BH^{(n-2)}$
since $BG$ is homotopy equivalent to a complex of dimension $cd(G)\le n-2$.
Let $H:BG\times I\to BH$ be a corresponding homotopy. It defines a homotopy $\hat H:S(BG)\times I\to S(BH)$
and, hence a homotopy $\tilde H:BH\times I\to S(BH)$ of $S(p)\circ\hat p$.
Since $\tilde H(BH\times\{0\})\subset S_0(BH)$, it follows that $\tilde H(BH\times I)\subset S_0(BH)$.
Thus, $\tilde H$ is a homotopy of $j_H$ to a map $g$ with image in $S_0(BH^{(n-2)})$. Hence we have a homotopy commutative diagram
$$
\begin{CD}
BH @<id<< BH\\
@VgVV @Vj_HVV\\
S_0(BH^{(n-2)}) @>\subset>> S_0(BH)\\
\end{CD}
$$
which defines a commutative diagram of homology groups
$$
\begin{CD}
H_n(BH;\Q) @<id_*<< H_n(BH;\Q)\\
@Vg_*VV @V\cong VV\\
H_n(S_0(BH^{(n-2)});\Q) @>0>> H_n(S_0(BH);\Q).\\ 
\end{CD}
$$
The right vertical arrow is an isomorphism in view of  Theorem~\ref{strings} (a). By Theorem~\ref{strings}(b), the commutativity of the diagram implies
that $H_n(BH;\Q)=0$. This brings a contradiction.
\end{proof}

{\bf Proof of Theorem~\ref{dm}.}
The condition $cd(G)<\infty$ implies that $cd(G_x)<\infty$ for all $x\in G$. By Theorem~\ref{main} there is $N\in\N$ such that
$H_m(N_x;\Q)=0$ for $m>N$  all $x\in G$. Therefore, the short exact sequence for $T_*(x,\Q)$,
(see the introduction) implies $T_*(x,\Q)=0$. \qed

\


\begin{thebibliography}{[CJY]}




\bibitem[Ba]{Ba} H. Bass, {\em Euler Characteristics and Characters of Discrete Groups}, Invent. math. 35 (1976), 155-196.



\bibitem [Bo]{Bo} A. Borel, {\em Seminar on transformation groups}, Annals of mathematical Studies, vol 46, Princeton University Press, 1960.


\bibitem[Bre]{Bre} G. Bredon,  Sheaf theory, Second edition. Graduate Texts in Mathematics, 170. Springer-Verlag, New York, 1997.

\bibitem[Br]{Br} K. Brown,  Cohomology of groups, Springer 1982.

\bibitem[Bu]{Bu} D. Burghelea, {\em The cyclic homology of the group rings}, Comment. Math. Helvetici
60 (1985), 354-365.

\bibitem[CJY]{CJY} R. Cohen, J. Jones, and J. Yan, {\em The loop homology algebra of spheres and projective spaces},
Categorical decomposition techniques in algebraic topology (Isle of Skye, 2001), 77–92, Progr. Math., 
215, Birkhäuser, Basel, 2004. 

\bibitem[C]{C}
Connes, Alain {\em Cohomologie cyclique et foncteurs Extn. (French) [Cyclic cohomology and functors Extn]} C. R. Acad. Sci. Paris Sér. I Math. 296 (1983), no. 23, 953–958. 



\bibitem[E]{E} B. Eckmann, {\em Cyclic homology of groups and the Bass conjecture}, Comment.
Math. Helvetici 61 (1986), 193-202.


\bibitem[Em]{Em} I. Emmanouil, {\em On class of groups satisfying Bass' conjecture}, Invent. math. 132, (1998) 307-330.

\bibitem[EM]{EM} A. Engel, M. Marcinkowski, {\em Burghelea conjecture and asymptotic dimension of groups}. Preprint (2016)  arXiv:1610.10076.


\bibitem[J]{J}
R. Ji, {\em Nilpotency of Connes' Periodicity Operator and the Idempotent Conjectures}, K-Theory 9 (1995), 59-76.

\bibitem[Li]{Li}
P. A.  Linnel, {\em Decomposition of augmentation ideals and relation mudules}. Proc. London Math. Soc. (3) 47, (1983)  83-127.
 

\bibitem[L]{L}
Loday, Jean-Louis, Cyclic homology.  Springer-Verlag, Berlin, 1992. 

\bibitem[LQ]{LQ} Loday, Jean-Louis; Quillen, Daniel {\em Homologie cyclique et homologie de l'algèbre de Lie des matrices. (French) [Cyclic homology and homology of the Lie algebra of matrices]} C. R. Acad. Sci. Paris Sér. I Math. 296 (1983), no. 6, 295–297.

 \bibitem[ML]{ML}
S. MacLane,  Homology. Reprint of the first edition. Die Grundlehren der mathematischen Wissenschaften, Band 114. Springer-Verlag, Berlin-New York, 1967.

\bibitem[Mo]{Mo} Morita Shigeyuki, Geometry of differential forms. AMS 2001.

\bibitem[R]{R} Rosenberg, Jonathan Algebraic K-theory and its applications. Graduate Texts in Mathematics, 147. Springer-Verlag, New York, 1994. x+392 pp.




\bibitem[T]{T}
Tsygan, B. L. {\em Homology of matrix Lie algebras over rings and the Hochschild homology}. (Russian) Uspekhi Mat. Nauk 38 (1983), no. 2(230), 217-218.



\bibitem[We]{We}
C. Westerland, {\em String homology of spheres and projective spaces}  Algebr. Geom. Topol. 7 (2007), 309-325. 








\end{thebibliography}
\end{document}